\documentclass[12pt]{amsart}
\usepackage{amssymb}
\usepackage{amsthm}

\newtheorem{theorem}{Theorem}[section]

\newtheorem{corollary}[theorem]{Corollary}

\theoremstyle{definition}
\newtheorem{definition}[theorem]{Definition}
\newtheorem{example}[theorem]{Example}
\newtheorem{question}[theorem]{Question}
\newtheorem{remark}[theorem]{Remark}

\newcommand{\U}{\mathcal{U}}

\begin{document}
	\title{Some properties of selectively star-ccc spaces}
	\author{Yuan Sun$^{*}$}
	\address{College of Scienc, Beijing University of Civil Engineering and Architecture, Beijing 102616, China}
	\email{sunyuan@bucea.edu.cn}
	
\begin{abstract}
		In 2013 \cite{aur}, Aurichi introduced a topological property named selectively ccc that can be viewed as a selective version of the countable chain condition (CCC). Later, Bal and Kočinac in \cite{bal} extended Aurichi’s work and defined the star version of the selectively ccc property called selectively $k$-star-ccc. The aim of this paper is twofold. Firstly, we establish connections between the selectively $k$-star-ccc properties, the chain conditions and other star-Lindelöf properties. Secondly, some examples are presented to solve questions raised by Xuan and Song in \cite{xuan-2}.
	\end{abstract}
	
	
	
	\maketitle

	\section{Introduction}
	
Throughout this paper, let $\mathbb{N}^+$ be the set of positive integers. Let $\omega$ and $\omega_1$ denote respectively the first infinite ordinal and the first uncountable ordinal. 

Let us first recall some notations and definitions of the star covering properties that can be found in \cite{van}. Let $\U$ be an open cover of a topological space $X$ and let $A$ be a subset of $X$. We denote $\mathrm{ST}^{1}\left(A, \mathcal{U}\right)=\{U \in \mathcal{U}: U\cap A \neq \emptyset\}$ and let $\mathrm{st}^{1}(A, \mathcal{U})=\bigcup \mathrm{ST}^{1}(A, \mathcal{U})$. Inductively, we denote $\mathrm{ST}^{n+1}(A,\mathcal{U})=   \left\{U\in \mathcal{U}: U \cap \mathrm{st}^{n}(A, \mathcal{U})\neq \emptyset\right\}$ and let $\mathrm{st}^{n+1}(A,\mathcal{U})=\bigcup \mathrm{ST}^{n+1}(A,\mathcal{U})$. For brevity we will write $\mathrm{ST}(A,\mathcal{U})$ and $\operatorname{st}(A,\mathcal{U})$ respectively for $n=1$, and write $\mathrm{st}\left(x,\mathcal{U}\right)$ for $\mathrm{st}\left(\{x\},\mathcal{U}\right)$. 

\begin{definition}\label{df1.1}
A space $X$ has the \emph{countable chain condition} (CCC)
if any pairwise disjoint open family in it is countable. A space $X$ is said to have the \emph{discrete countable chain condition} (DCCC) provided there does not exist an uncountable discrete open family.
	\end{definition}

\begin{definition}\label{df1.2}
For $k\in\mathbb{N}^+$. A space $X$ is said to be \emph{$k$-star-Lindelöf} if for every open cover $\mathcal{U}$ of $X$ there is some countable subfamily $\mathcal{U}'\subset\mathcal{U}$ such that $\operatorname{st}^k\left(\bigcup\mathcal{U}',\mathcal{U}\right)=X$. For brevity we will write $k$-SL for $k$-star-Lindelöf and SL for $k=1$.
	\end{definition}

\begin{definition}\label{df1.3}
 For $k\in\mathbb{N}^+$. A space $X$ is said to be \emph{strongly $k$-star-Lindelöf} (denoted by $k$-SSL) if for every open cover $\mathcal{U}$ of $X$, there is some countable subset $A\subset X$ such that $\operatorname{st}^k(A,\mathcal{U})=X$.
	\end{definition}

By definitions, it is easily seen that every Lindelöf space is SSL. More generally, every $k$-SSL space is $k$-SL and every $k$-SL space is $k+1$-SSL. Note that the SL property (i.e., star-Lindelöf) defined in \cite{xuan-2} is stronger than here: in that paper, a space $X$ is called star-Lindelöf if for any open cover $\mathcal{U}$ of $X$, there is a Lindelöf subset $A\subset X$ such that $\mathrm{st}(A,\mathcal{U}) =
X$. 

\begin{definition}\label{df1.4}
A space is said to be \emph{weakly Lindelöf} if for every open cover $\mathcal{U}$ of $X$, there is some countable subfamily $\mathcal{U}'\subset\mathcal{U}$ whose union is dense in $X$, i.e., $\overline{\bigcup\mathcal{U}'}=X$.
	\end{definition}

Clearly, every Lindelöf space is weakly Lindelöf and every weakly Lindelöf space is SL. 

\begin{definition}\label{df1.5}
For $k\in\mathbb{N}^+$. A space $X$ is said to be \emph{weakly $k$-star-Lindelöf} (denoted by weakly $k$-SL) if for every open cover $\mathcal{U}$ of $X$, there is some countable subfamily $\mathcal{U}'\subset\mathcal{U}$ such that $\overline{\mathrm{st}^k\left(\bigcup\mathcal{U}',\mathcal{U}\right)}=X$. Similarly, a space $X$ is called \emph{weakly strongly $k$-star-Lindelöf} (weakly $k$-SSL) if for every open cover $\mathcal{U}$ of $X$, there is some countable subset  $A\subset X$ such that $\overline{\operatorname{st}^k(A,\mathcal{U})}=X$.
	\end{definition}

Obviously, every weakly $k$-SL space is $k+1$-SL and every weakly $k$-SSL space is $k+1$-SSL for $k\in\mathbb{N}^+$. Note also that the weakly SSL property (i.e., weakly strongly star-Lindelöf) defined here was named \emph{weakly star countable} in \cite{ala, lin, xuan-2}. 

\begin{definition}\label{df1.6}
A space $X$ is said to be \emph{$\omega$-star-Lindelöf} (denoted by $\omega$-SL) if for every open cover $\mathcal{U}$ of $X$, there is some $k\in\mathbb{N}^+$ and some countable subfamily $\mathcal{U}'\subset \mathcal{U}$ such that $\operatorname{st}^k\left(\bigcup\mathcal{U}',\mathcal{U}\right)=X$. 
	\end{definition}

By definition, it is easy to see that the $\omega$-SL property is a generalization of $k$-SL and of $k$-SSL for each $k\in\mathbb{N}^+$. Moreover, by using the $\omega$-SL property, van Douwen et al., established connections between the various star properties and the chain conditions. The following results can be found in \cite{van}.

\begin{theorem}\label{th1.7}
\begin{enumerate}
\item Every CCC space is weakly Lindelöf and every DCCC space is weakly SL.
\item In regular spaces, the $\omega$-SL property implies the DCCC; thus the DCCC, equals weakly SL, equals $2$-SL, equals $\omega$-SL, and all the properties in between.
\item In normal spaces, the DCCC implies the weakly SSL property; hence the DCCC equals weakly SSL, equals $2$-SSL, equals $\omega$-SL, and all the properties in between.
\item In perfectly normal spaces, the CCC, the weakly Lindelöf property, the SL property and all the properties in between are equivalent.
\end{enumerate}
\end{theorem}

In 2013 \cite{aur}, Aurichi presented a natural way of defining a selective version of the CCC.

\begin{definition}\label{df1.8}
A space $X$ is said to be \emph{selectively ccc} if for every sequence $(\mathcal{A}_n:n\in\omega)$ of maximal pairwise disjoint open families in $X$, there exists a sequence $(A_n\in\mathcal{A}_n:n\in\omega)$ whose union is dense in $X$, i.e., $\overline{\bigcup_{n\in\omega} A_n}=X$. 
	\end{definition}

Later, Bal and Kočinac in \cite{bal} extend Aurichi’s work and defined the star version of the selectively ccc property.

\begin{definition}\label{df1.9}
Let $k\in\mathbb{N^+}$. A space $X$ is said to be a \emph{selectively $k$-star-ccc} space if for every open cover $\mathcal{U}$ of $X$ and every sequence $(\mathcal{A}_n:n\in\omega)$ of maximal pairwise disjoint open families of $X$, there is a sequence $(A_n\in\mathcal{A}_n:n\in\omega)$ such that $\operatorname{st}^k(\bigcup_{n\in\omega} A_n,\mathcal{U})=X$. We write selectively star-ccc instead of selectively $1$-star-ccc.
	\end{definition}

In several papers \cite{song-1,song-2,xuan-1,xuan-2}, Song and Xuan investigated the relations between the selectively $k$-star-ccc properties and the chain conditions.

\begin{theorem}\label{th1.10}

\begin{enumerate}
\item Every CCC space is selectively $2$-star-ccc \cite[Corollary~3.2]{xuan-2}.

\item The selectively star-ccc property implies both the DCCC \cite[Theorem~3.6]{xuan-1} and the weakly SSL property \cite[Theorem~3.12]{xuan-2}. However, there exists a Tychonoff DCCC space that is not selectively star-ccc \cite[Example~3.7]{xuan-1}. 

\item There exists a Tychonoff selectively $2$-star-ccc space that is neither SSL nor selectively star-ccc \cite[Example~3.3]{song-1}.

\item There exist counterexamples showing that both the product of a selectively star-ccc space with a Lindelöf space \cite[Example~3.11]{song-2} or the product of two Lindelöf spaces \cite[Example~3.16]{xuan-2} need not be selectively star-ccc.
\end{enumerate}
\end{theorem}

In this paper, we will continue to investigate the properties of selectively $k$-star-ccc spaces. We mainly prove that:

\begin{enumerate} 
\item In regular spaces, the DCCC equals selectively $3$-star-ccc and all $k$-star-ccc in between (see Corollary~\ref{co2.6} below). However, there exists a Hausdorff space that is selectively $3$-star-ccc but does not have the DCCC (see Example~\ref{ex3.4} below). Therefore, we could partially answer Question~4.1 in \cite{xuan-2}.

\item In normal spaces, the selectively $3$-star-ccc property and the weakly SSL property are equivalent (see Corollary~\ref{co2.6} below). Thus, we could answer Question ~4.10 in \cite{xuan-2}.

\item There exists a first-countable space that is selectively $3$-star-ccc but is neither weakly SSL nor selectively $2$-star-ccc (see Example~\ref{ex3.4} below); that could give Question 4.9 and 4.11 in \cite{xuan-2} a negative answer.

\item The product of a selectively star-ccc space with a compact space need not be selectively star-ccc (see Example~\ref{ex3.5} below). 
\end{enumerate}

\section{General results}

\begin{theorem}\label{th2.1}
If $X$ is a selectively star-ccc space, then $X$ is SL. In general, every selectively $k$-star-ccc space is $k$-SL for each $k\in \mathbb{N}^+$.
\end{theorem}
\begin{proof}
Suppose $X$ is not $k$-SL and let $\mathcal{U}$ be an open cover of $X$ such that if $\{U_n:n\in\omega\}\subset \mathcal{U}$ is countable, then $\operatorname{st}^k(\bigcup_{n\in\omega}U_n,\mathcal{U})\neq X$. Let $\mathcal{A}$ be a maximal pairwise disjoint open family which refines $\mathcal{U}$ (in fact, every open dense family admits a refinement that is a maximal pairwise disjoint open family, see \cite{aur}). Define $\mathcal{A}_n=\mathcal{A}$ for every $n\in\omega$, hence $\left(\mathcal{A}_n: n\in \omega\right)$ is a maximal pairwise disjoint open family. Let $\left(A_n: n\in \omega\right)$ be any sequence with $A_n\in\mathcal{A}_n$. As $\mathcal{A}_n$ refine $\mathcal{U}$, there exists a countable subfamily $\{U_n:n\in\omega\}\subset \mathcal{U}$ such that $\bigcup_{n\in\omega}A_n\subset \bigcup_{n\in\omega}U_n$. However, by construction, $\operatorname{st}^k(\bigcup_{n\in\omega}U_n,\mathcal{U})$ cannot cover $X$, hence $\operatorname{st}^k(\bigcup_{n\in\omega} A_n,\mathcal{U})\neq X$, i.e., $X$ is not selectively $k$-star-ccc. 
	\end{proof}

\begin{theorem}\label{th2.2}
If $X$ is a SL space, then $X$ is selectively $2$-star-ccc. In general, every $k$-SL space is selectively $k+1$-star-ccc for each $k\in\mathbb{N}^+$.
\end{theorem}
\begin{proof}
We give the proof for $k= 1$. It is clear how the proof for other values of $k$ can be obtained. Suppose $X$ is SL and let $\mathcal{U}$ be any open cover of $X$, then there exists an countable subfamily $\{U_n:n\in\omega\}\subset \mathcal{U}$ such that $\operatorname{st}(\bigcup_{n\in\omega} U_n,\mathcal{U})= X$. Let $\left(\mathcal{A}_n: n\in \omega\right)$ be any maximal pairwise disjoint open family. For each $n\in\omega$, because $\bigcup\mathcal{A}_n$ is dense in $X$, then there exists $A_n\in\mathcal{A}_n$ such that $A_n\cap U_n\neq\emptyset$. Let $x$ be any point of $X$. Because $\operatorname{st}(\bigcup_{n\in\omega} U_n,\mathcal{U})= X$, then there exist $n(x)\in\omega$ and $U_x\in\mathcal{U}$ such that $x\in U_x$ and $U_x\cap U_{n(x)}\neq\emptyset$. Moreover, as $U_{n(x)}\cap A_{n(x)}\neq\emptyset$, it follows that $x\in\operatorname{st}^2(\bigcup_{n\in\omega} A_n,\mathcal{U})$ and hence $\operatorname{st}^2(\bigcup_{n\in\omega} A_n,\mathcal{U})= X$.
\end{proof}

Note that although Song and Xuan have proven that both weakly Lindelöf spaces and star-Lindelöf spaces are selectively $2$-star-ccc (see Theorem~3.1 and~3.5 in \cite{xuan-2}). However, the property of star-Lindelöf defined in that paper is stronger than here. So the above results are by no means trivial. 

\begin{corollary}\label{co2.3}
Every CCC space is selectively $2$-star-ccc.
\end{corollary}
\begin{proof}
For a CCC space $X$, any open cover has a countable subcover whose union is dense in $X$. Therefore every CCC space is weakly Lindelöf. Moreover, it is easy to see that every weakly Lindelöf space is SL. Thus by Theorem \ref{th2.2} above, it follows that every CCC space is selectively $2$-star-ccc.
\end{proof}

The above Corollary could answer Problem 4.5 raised by Bal and Kočinac in \cite{bal}. Note that such a result was first proved in \cite[Corollary 3.2]{xuan-2} by the direct implication: the CCC $\Longrightarrow $ the weakly Lindelöf property $\Longrightarrow $ the selectively $2$-star-ccc property.

Inspired by the definition of $\omega$–SL (see Definition~\ref{df1.6} above), we now consider the extension to selectively $\omega$-star-ccc.

\begin{definition}\label{df2.4}
A space $X$ is said to be \emph{selectively $\omega$-star-ccc} if for every open cover $\mathcal{U}$ of $X$ and every sequence $\left(\mathcal{A}_n: n\in \omega\right)$ of maximal pairwise disjoint open families, there exist a sequence $\left(A_n\in\mathcal{A}_n: n\in \omega\right)$ and a positive integer $k$ such that $\operatorname{st}^k(\bigcup_{n\in\omega} A_n,\mathcal{U})=X$.
\end{definition}

\begin{theorem}\label{th2.5}
The selectively $\omega$-star-ccc property is equivalent to the $\omega$-SL property.
\end{theorem}
\begin{proof}
Suppose $X$ is not $\omega$-SL, let $\mathcal{U}$ be an open cover of $X$ such that $\operatorname{st}^k(\bigcup_{n\in\omega}U_n,\mathcal{U})\neq X$ for any countable subfamily $\{U_n:n\in\omega\}\subset \mathcal{U}$ and positive integer $k$. Let $\mathcal{A}$ be a maximal pairwise disjoint open family that refines $\mathcal{U}$. Define $\mathcal{A}_n=\mathcal{A}$ for every $n\in\omega$, so $\left(\mathcal{A}_n: n\in \omega\right)$ is a sequence of maximal pairwise disjoint open families. Let $\left(A_n: n\in \omega\right)$ be any sequence with $A_n\in\mathcal{A}_n$. As $\mathcal{A}_n$ being a refinement of $\mathcal{U}$, there exists a countable subfamily $\{U_n:n\in\omega\}\subset \mathcal{U}$ such that $\bigcup_{n\in\omega}A_n\subset \bigcup_{n\in\omega}U_n$. However, by construction, $\operatorname{st}^k(\bigcup_{n\in\omega}U_n,\mathcal{U})$ cannot cover $X$, hence $\operatorname{st}^k(\bigcup_{n\in\omega} A_n,\mathcal{U})\neq X$. Because $\left(A_n: n\in \omega\right)$ and $k$ are arbitrary, thus $X$ is not selectively $\omega$-star-ccc.

Now suppose $X$ is $\omega$-SL. Let $\mathcal{U}$ be any open cover of $X$, then there exist a countable subfamily $\{U_n:n\in\omega\}\subset \mathcal{U}$ and some $k\in\mathbb{N}^+$ such that $\operatorname{st}^k(\bigcup_{n\in\omega} U_n,\mathcal{U})= X$. Let $\left(\mathcal{A}_n: n\in \omega\right)$ be any sequence of maximal pairwise disjoint open families. Because $\bigcup\mathcal{A}_n$ is dense in $X$ for every $n\in\omega$, then there exists $A_n\in\mathcal{A}_n$ such that $A_n\cap U_n\neq\emptyset$. Clearly, $\operatorname{st}^{k+1}(\bigcup_{n\in\omega} A_n,\mathcal{U})= X$. Therefore $X$ is selectively $\omega$-star-ccc.
\end{proof}

By Theorem \ref{th2.5} and Theorem\ref{th1.7}, we can get the following results.

\begin{corollary}\label{co2.6}

\begin{enumerate}
\item In regular spaces, every selectively $k$-star-ccc space has the DCCC for each $k\in\mathbb{N^+}$. In fact, the DCCC equals weakly SL, equals $2$-SL, equals selectively $3$-star-ccc, and all the properties in between.

\item  In normal spaces, the DCCC equals weakly SSL, equals $2$-SSL, equals selectively $3$-star-ccc, and all the properties in between.
\end{enumerate}
\end{corollary}

The above Corollary could answer Question~4.10 in \cite{xuan-2}. Moreover, in the next section, we will present a Hausdorff space that is selectively $3$-star-ccc (in fact is $2$-SL) but does not have the DCCC (see Example~\ref{ex3.4}). Therefore, we can also partially answer Question~4.1 in \cite{xuan-2}. However, the author does not know the answer to the following questions:

\begin{question}\label{co2.7}
Does there exist a selectively $2$-star-ccc Hausdorff space that does not have the DCCC.
\end{question}

\begin{question}[\cite{xuan-2}]\label{co2.8}
Is there a weakly SSL normal space which is not selectively $2$-star-ccc?
\end{question}

\begin{question}[\cite{xuan-2}]\label{co2.9}
Is there a selectively $3$-star-ccc normal space which is not selectively $2$-star-ccc?
\end{question}

\begin{remark}
Note that Question~2.8 and~2.9 are essentially equivalent by Corollary 2.6. Finding  counterexamples may be difficult: since both selectively $3$-star-ccc and weakly SSL are equivalent to $2$-SSL in normal spaces, while it is not even known whether there is a normal $2$-SSL space
which is not SL (the question was first raised in \cite{van}).
\end{remark}

\section{Some examples}
In this section we will give some examples to make distinctions.
	
\begin{example}\label{ex3.1}
The ordinal space $[0,\omega_1)$ has the following properties:

\begin{enumerate}
\item It is selectively star-ccc but not selectively ccc;
\item It is a DCCC space but not a CCC space;
\item It is SSL (since it is countably compact) but not weakly Lindelöf, and hence not Lindelöf.
\end{enumerate}
\end{example}
	
	\begin{proof}
		By using a concept named \emph{absolutely compact} due to Matveev (see \cite{mat} for more details), Song and Xuan in \cite{song-2} showed that $[0,\omega_1)$ being a absolutely compact space is selectively star-ccc. Because $[0,\omega_1)$ is countably compact, it is easy to see that $[0,\omega_1)$ is both DCCC and SSL. The remaining implications are straightforward.
	\end{proof}

 \begin{example}\label{ex3.2}
	The Isbell-Mrówka space $\Psi(\omega)$ is selectively star-ccc but is neither Lindelöf nor countably compact.
\end{example}
	
	\begin{proof}
We first give the definition of $\Psi(\omega)$. Recall that a family $\mathcal{A}$ of infinite subsets of a set $X$ is called \emph{almost disjoint} if
the intersection of any two distinct sets in $\mathcal{A}$ is finite. Let $\{N_s:s\in S\}$ be a maximal almost disjoint family of $\omega$. Let $\Psi(\omega)=\omega\cup S$ and topologize $\Psi(\omega)$ as follows: points of $\omega$ are isolated, basic neighborhoods of points $s\in S$ take the form $\{s\} \cup\left(N_{s}-F\right)$ where $F\subset \omega$ is finite. It is straightforward to verify that $\Psi(\omega)$ is Hausdorff, first countable, separable and locally compact. However, because $S$ is an uncountable closed discrete subset, thus $\Psi(\omega)$ is neither Lindelöf nor countably compact. This is a classic example due to Mrówka \cite{mro} and Isbell. Because $\omega$ is a countable dense subset consisting of isolated points, thus $\Psi(\omega)$ is selectively ccc and hence selectively star-ccc. Moreover, since $S$ is an uncountable closed discrete subset, $\Psi(\omega)$ is neither Lindelöf nor countably compact.
	\end{proof}
	
 \begin{example}\label{ex3.3}
There exists a Hausdorff space that is SSL but not selectively star-ccc.
\end{example}	
 \begin{proof}
Let $X$ be a set constructed from the ordinal space $[0,\omega_1]$ by placing between each ordinal $\alpha$ and its successor $\alpha+1$ a copy of the unit interval
$I= (0, 1)$. We denote by $(X, \tau)$ the space $X$ equipped with the order topology $\tau$. Note that $(X, \tau)$ is Hausdorff and compact.

For each $n\in\mathbb{N}^+$, we express $(X,\tau)$ as the union of pairwise disjoint dense subsets $A_1, A_2,\ldots, A_{2n+1}$ such that both $0$ and all limit ordinals are in $A_{2n+1}$. Let $E_{2i+1}=A_{2i}\cup A_{2i+1}\cup A_{2i+2}$ for $i=0,1,\ldots,n$ and $E_{2i}=A_{2i}$ for $i=1,\ldots,n$, where $A_0= A_1$ and $A_{2n+2}=A_{2n+1}$. Thus we can define a new topology $\tau_n$ on $X$ as follows: a basic open set takes the form $I_x\cap E_{n(x)}$, where $I_x$ containing $x$ is some interval in the order topology $\tau$ and $n(x)$ is the unique integer with $x\in A_{n(x)}$. Clearly, all $(X, \tau_n)$ are Hausdorff. Such a construction was first described by Sarkhel in \cite{sar}. Later, van Douwen et al., in \cite{van} proved that $(X, \tau_n)$ is $n$-SL but not $n$-SSL. Indeed, they showed that there exists an uncountable pairwise disjoint family $\{(x_{\alpha},y_{\alpha}): \alpha<\omega_1\}$ of nonempty intervals of $X$ satisfying:

\begin{enumerate}
\item $x_0=0$ and $x_{\alpha}\in A_{2n+1}$ for each $\alpha<\omega$;

\item $y_{\alpha}\in A_{2n+1}$ such that $x_{\alpha}<y_{\alpha}\leqslant x_{\alpha+1}$, where $y_{\alpha}$ is not in the $(0, 1)$ segment containing or immediately following $x_{\alpha}$;

\item $\bigcup_{i=1}^{2 n} A_{i}\subset \bigcup_{\alpha<\omega_1}(x_{\alpha},y_{\alpha})$.
\end{enumerate}

 Moreover, they use a similar method (with $X =A_1\cup\cdots\cup A_{2n}$) to get a topology (denoted by $\tau_{n}^{\prime}$) such that $(X,\tau_{n}^{\prime})$ is $n$-SSL but not $n-1$-SL. The above proofs can be found in \cite[Theorem 2.1.5]{van}. We now prove that $(X,\tau'_1)$ is not selectively star-ccc.
 
Let $\mathcal{A}=\{(x_{\alpha}, y_{\alpha}):\alpha<\omega_1\}$ be the uncountable pairwise disjoint open family stated above and let $\mathcal{A}_n=\mathcal{A}$ for each $n\in\omega$. It is easy to see that $(\mathcal{A}_n:n\in\omega)$ is a sequence of maximal pairwise disjoint open families. Let $\mathcal{U}=\{E_2\}\cup\{(x_{\alpha}, y_{\alpha}): \alpha<\omega_1\}$, then $\mathcal{U}$ is an open cover of $(X,\tau_1')$. For any sequence $((x_n,y_n): n\in\omega)$ with $(x_n,y_n)\in\mathcal{A}_n$. As $\mathcal{A}$ is uncountable, there exists some $\alpha$ such that $(x_{\alpha}, y_{\alpha})\cap\bigcup_{n\in\omega}(x_n,y_n)=\emptyset$. Hence there exists a point $a_{\alpha}\in A_1\cap(x_{\alpha}, y_{\alpha})$ such that $\operatorname{st}(a_{\alpha}, \mathcal{U})\cap\bigcup_{n\in\omega}(x_n,y_n)=\emptyset$. Therefore, $\operatorname{st}(\bigcup_{n\in\omega}(x_n,y_n), \mathcal{U})\neq X$ and hence $(X, \tau_1')$ is not selectively star-ccc.
	\end{proof}
 
Next, we will use the same technique, but for the ordinal space $[0,\omega_1)$, to construct some counterexamples.  
 
  \begin{example}\label{ex3.4}
There exists a Hausdorff and first-countable space that has the SL property but is not selectively star-ccc. Moreover, there exists a Hausdorff and first-countable space that is $2$-SL but not selectively $2$-star-ccc and does not have the DCCC.
\end{example}

\begin{proof}
Let $Y$ be the ordinal space $[0,\omega_1)$ equipped with the topology $\tau$ or $\tau_n$ constructed in Example \ref{ex3.3}. 

\medskip
{\bf Claim~1.} ~\emph{Both $(Y,\tau)$ and $(Y, \tau_n)$ are first-countable}. For $(Y,\tau)$, if point $x\in Y$ either belongs to some interval $(0,1)$ or is a non-limit ordinal, then countable family $\{(x-\frac{1}{n}, x+\frac{1}{n}):n\in\mathbb{N}^+\}$ is the neighbourhood base of $x$. On the other hand, the countable family $\{[\alpha, x+\frac{1}{n}):\alpha\in[0,x), n\in\mathbb{N}^+\}$ is the neighbourhood base for other $x\in Y$. Similarly, it is easy to see that $\{(x-\frac{1}{i}, x+\frac{1}{i})\cap E_{n(x)}:i\in\mathbb{N}^+\}$ or $\{[\alpha, x+\frac{1}{i})\cap E_{n(x)}:\alpha\in[0,x), i\in\mathbb{N}^+\}$ is a countable neighbourhood base of $x\in (Y, \tau_n)$.

\medskip
{\bf Claim 2.}~\emph{$(Y, \tau_1)$ is a SL space (hence selectively $2$-star-ccc) but not selectively star-ccc}. Let $\mathcal{U}=\{I_x\cap E_{n(x)}: x\in Y\}$ be any basic open cover of $(Y, \tau_1)$. Because $(Y, \tau)$ is SL and $\mathcal{I}=\{I_x:x\in  Y\}$ is an open cover of $(Y, \tau)$, thus there exists a countable subfamily $\{I_{x_i}: i\in\omega\}\subset\mathcal{I}$ such that $\operatorname{st}(\bigcup_{i\in\omega}I_{x_i}, \mathcal{I})=Y$; let $\mathcal{U}'=\{I_{x_i}\cap E_{n(x_i)}: i\in\omega\}$ be the corresponding countable subfamily of $\mathcal{U}$. For any $x\in Y$, since $\operatorname{st}(\bigcup_{i\in\omega}I_{x_i}, \mathcal{I})=Y$, there exists some integer $i$ such that $I_x\cap I_{x_i}\neq\emptyset$. Moreover, because the set $E_{n(x)}\cap E_{n(x_i)}$ contains at least the dense subset $A_2$ of $(Y,\tau)$, thus $\big(I_{x_i}\cap E_{n(x_i)}\big)\cap \big(I_{x}\cap E_{n(x)}\big)\neq\emptyset$, and hence $\operatorname{st}(\bigcup\mathcal{U}', \mathcal{U})=Y$, i.e., $(Y,\tau_1)$ is SL.

Let $\mathcal{A}=\{(x_{\alpha}, y_{\alpha}):\alpha<\omega_1\}$ be the uncountable pairwise disjoint open family stated in Example \ref{ex3.3} and let $\mathcal{A}_n=\mathcal{A}$ for each $n\in\omega$. Let $\mathcal{V}=\{E_3\}\cup\{(x_{\alpha}, y_{\alpha}): \alpha<\omega_1\}$, then $\mathcal{V}$ is an open cover of $(Y,\tau_1)$. For any sequence $((x_n,y_n)\in\mathcal{A}_n: n\in\omega)$, it is easy to see that there exists some $\alpha$ such that $(x_{\alpha}, y_{\alpha})\cap\bigcup_{n\in\omega} (x_n,y_n)=\emptyset$. Hence there is a point $a_{\alpha}\in A_1\cap(x_{\alpha}, y_{\alpha})$ such that $\operatorname{st}(a_{\alpha}, \mathcal{V})\cap\bigcup_{n\in\omega}(x_n,y_n)=\emptyset$. Therefore, $\operatorname{st}(\bigcup_{n\in\omega}(x_n,y_n), \mathcal{V})\neq Y$ and hence $(Y, \tau_1)$ is not selectively star-ccc.

\medskip
{\bf Claim 3.}~\emph{$(Y, \tau_2)$ is $2$-SL (hence selectively $3$-star-ccc) but is neither selectively $2$-star-ccc nor weakly SSL}. Let $\mathcal{U}=\{I_x\cap E_{n(x)}: x\in Y\}$ be any basic open cover of $(Y, \tau_2)$. Because $(Y, \tau)$ is SL and $\mathcal{I}=\{I_x:x\in  Y\}$ is an open cover of $(Y, \tau)$, then there exists an countable subfamily $\{I_{x_i}: i\in\omega\}\subset\mathcal{I}$ such that $\operatorname{st}(\bigcup_{i\in\omega}I_{x_i}, \mathcal{I})=Y$; let $\mathcal{U}'=\{I_{x_i}\cap E_{n(x_i)}: i\in\omega\}$ be the corresponding countable subfamily of $\mathcal{U}$. For any $x\in Y$, since $\operatorname{st}(\bigcup_{i\in\omega}I_{x_i}, \mathcal{I})=Y$, there exists some $x_i$ such that $I_x\cap I_{x_i}\neq\emptyset$. Without loss of generality, we assume that $x_i\in A_1$ and $x\in A_5$. Because $I_x\cap I_{x_i}\neq\emptyset$ and $A_3$ is dense in $(Y,\tau)$, so we can pick a point $y\in I_x\cap I_{x_i}\cap A_3$. Moreover, since $E_5\cap E_3$ and  $E_3\cap E_1$ contain respectively the dense subsets $A_4$ and $A_2$, it follows that $\big(I_{x}\cap E_{5}\big)\cap \big(I_{y}\cap E_3\big)\neq\emptyset$ and $\big(I_{y}\cap E_{3}\big)\cap \big(I_{x_i}\cap E_1\big)\neq\emptyset$. Thus $x\in \operatorname{st}^2(\bigcup\mathcal{U}', \mathcal{U})$ and hence $\operatorname{st}^2(\bigcup\mathcal{U}', \mathcal{U})=Y$. Therefore, the space $(Y,\tau_2)$ is $2$-SL.

Let $\mathcal{A}=\{(x_{\alpha}, y_{\alpha}):\alpha<\omega_1\}$ be the uncountable pairwise disjoint open family let $\mathcal{A}_n=\mathcal{A}$ for each $n\in\omega$. Let $\mathcal{V}=\{E_5\}\cup\{(x_{\alpha}, y_{\alpha})\cap E_i: \alpha<\omega_1, i=1,2,3,4\}$, then $\mathcal{V}$ is an open cover of $(Y,\tau_2)$. Let $((x_n,y_n)\in\mathcal{A}_n:n\in\omega)$ be any sequence, then there exists $\alpha$ such that $a_{\alpha}\in (x_{\alpha}, y_{\alpha})\cap A_1$ and $(x_{\alpha}, y_{\alpha})\cap\bigcup_{n\in\omega}(x_n,y_n)=\emptyset$. But $\left(x_{\alpha}, y_{\alpha}\right) \cap E_{1}$ is the only member
of $\mathcal{V}$ containing $a_{\alpha}$, thus $\mathrm{st}\left(a_{\alpha}, \mathcal{V}\right) =\left(x_{\alpha}, y_{\alpha}\right) \cap E_{1}$ and hence $\mathrm{st}^2\left(a_{\alpha}, \mathcal{V}\right) \subset\left(x_{\alpha}, y_{\alpha}\right) \cap\left(E_{1} \cup E_{2}\cup E_3\right) \subset\left(x_{\alpha}, y_{\alpha}\right)$. Therefore, there exist an open cover $\mathcal{V}$ and a sequence $\left(\mathcal{A}_n: n\in \omega\right)$ of maximal pairwise disjoint open families such that for any sequence $((x_n,y_n)\in\mathcal{A}_n:n\in\omega)$, we have $\mathrm{st}^2(\bigcup_{n\in\omega}(x_n,y_n),\mathcal{V})\neq Y$. Thus $(Y,\tau_2)$ is not selectively $2$-star-ccc.

Finally, we prove that $(Y, \tau_2)$ is not weakly SSL. Consider the open cover $\mathcal{V}=\{E_5\}\cup\left\{(x_{\alpha}, y_{\alpha})\cap E_i: \alpha<\omega_1, i=1,2,3,4\right\}$. Let $B$ be any countable subset of $Y$, then there exists some $\alpha<\omega_1$ such that $\mathrm{st}(B,\mathcal{V})\subset E_5\cup\bigcup_{\beta\leqslant\alpha}\left(x_{\beta}, y_{\beta}\right)$. Pick any $a_{\alpha+1}\in(x_{\alpha+1}, y_{\alpha+1})\cap A_1$. Clearly, $(x_{\alpha+1}, y_{\alpha+1})\cap E_1$ is an open subset containing $a_{\alpha+1}$ but cannot intersect $E_5\cup\bigcup_{\beta\leqslant\alpha}\left(x_{\alpha}, y_{\alpha}\right)$. Thus $\overline{\mathrm{st}(B,\mathcal{V})}\neq Y$ and hence $(X,\tau_2)$ is not weakly SSL. Note that $\{(x_{\alpha}, y_{\alpha})\cap E_1: \alpha<\omega_1\}$ is an uncountable discrete open family of $(Y,\tau_2)$, so it also does not have the DCCC.
	\end{proof}
 
 The above discussions show that the space $(Y,\tau_2)$ is first-countable and selectively $3$-star-ccc but is neither weakly SSL nor selectively $2$-star-ccc, which give a negative answer to Question 4.9 and 4.11 in \cite{xuan-2}. 

Song and Xuan presented an example \cite[Example 3.11]{song-2} to show that the product of a selectively star-ccc space with a Lindelöf space need not be selectively star-ccc. Moreover, in \cite[Example 3.16]{xuan-2}, they even present a space showing that the product of two Tychonoff Lindelöf spaces need not be selectively star-ccc. The following example shows that the product of a selectively star-ccc space with a compact space also need not be selectively star-ccc. Note that such an example was first described in \cite{van} to show that the product of a SSL space and a compact space need not be SSL. We give the proof for the convenience of the reader.

 \begin{example}\label{ex3.5}
There exists a space that is not selectively star-ccc, despite being the product of a selectively star-ccc space and a compact space.
\end{example}
 
 \begin{proof}
 Let $X=\Psi(\omega)$ be the Isbell-Mrówka space stated in Example \ref{ex3.2}; index $S$ as $\left\{s_{\alpha}:\alpha<\kappa\right\}$. As $\omega$ is a countable dense subset consisting of isolated points, thus $X$ is selectively star-ccc. Let $\left\{y_{\alpha}:\alpha<\kappa\right\}$ be a discrete space with cardinality $\kappa$ and let $Y=\left\{y_{\alpha}:\alpha<\kappa\right\}\cup \left\{\infty \right\}$ be the one-point compactification. 

We now show that $X\times Y$ is not selectively star-ccc. The family $\mathcal{U}=\left\{X\times\left\{y_{\alpha}\right\}:\alpha<\kappa\right\}\cup\left\{N_{s_{\alpha}}\times Y-\{y_{\alpha}\}:\alpha<\kappa\right\}\cup\left\{\{n\}\times Y: n\in \omega\right\}$ is an open cover of $X\times Y$ such that if point $(s_{\alpha},y_{\alpha})\in U\in\mathcal{U}$, then $U$ is of the form $X\times \{y_{\alpha}\}$. Because the subset $\omega\times \left\{y_{\alpha}:\alpha<\kappa\right\}$ is dense and each point $(n,y_{\alpha})$ in it is isolated, so $\mathcal{A}=\left\{\{(n,y_{\alpha})\}:n\in\omega, \alpha<\kappa\right\}$  is a maximal pairwise disjoint open family. Let $\mathcal{A}_n=\mathcal{A}$ for each $n\in\omega$. If $(A_n:n\in\omega)$ is a sequence with $A_n\in\mathcal{A}_n$, then there exists $\alpha$ such that $X\times\{y_{\alpha}\}\cap\bigcup_{n\in\omega}A_n=\emptyset$. Therefore $(s_{\alpha},y_{\alpha})\notin\operatorname{st}\left(\bigcup_{n\in\omega}A_n,\mathcal{U}\right)$ and hence $X\times Y$ is not selectively star-ccc.
 \end{proof}


\begin{thebibliography}{11}

        \bibitem{mro} 
        S. Mrowka, On completely regular spaces, Fundam. Math. 41 (1954) 105-106.
        \bibitem{sar} 
        D.N. Sarkhel, Some generalizations of countable compactness, Indian J. Pure Appl. Math. 17 (1986) 778-785.
        \bibitem{van}
		E.K. van Douwen, G.M. Reed, A.W. Roscoe, I.J. Tree, Star covering properties, Topol. Appl. 39 (1991) 71-103.
  
        \bibitem{mat}
        M.V. Matveev, On absolutely countably compact spaces, Topol. Appl. 58 (1994) 81–91.
		
		\bibitem{aur}
		L.F. Aurichi, Selectively c.c.c. spaces, Topol. Appl. 160 (2013) 2243–2250.
		
		
		\bibitem{ko} \label {ko}
		Lj.D.R. Kočinac, Star selection principles: a survey, Khayyam J. Math. 1 (2015) 82–106.
  
        \bibitem{ala} \label {ala}
        O.T. Alas, R.G. Wilson, Properties related to star countability and star finiteness. Topol. Appl. 221 (2017) 432–439.

       
       \bibitem{lin} \label {lin}
        Shou Lin, Jinjin Li, Zhangyong Cai, The establishment and development of star-Lindelöf spaces, Topol. Appl. 283 (2020) 107341.
  
		
		\bibitem{bal} \label {bal}
		P. Bal, Lj.D.R. Kočinac, On selectively star-ccc spaces, Topol. Appl. 281 (2020) 107184.
		
		\bibitem{song-1} \label {song-1}
		Yan-Kui Song, Wei-Feng Xuan, A note on selectively star-ccc spaces, Topol. Appl. 263 (2019) 343–349.
		
		\bibitem{song-2} \label {song-2}
		Yan-Kui Song, Wei-Feng Xuan, More on selectively star-ccc spaces, Topol. Appl. 268 (2019) 106905.
		
		\bibitem{xuan-1} \label {xuan-1}
		Wei-Feng Xuan, Yan-Kui Song, A study of selectively star-ccc spaces, Topol. Appl. 273 (2020) 107103.
		
		\bibitem{xuan-2} \label {xuan-2}
		Wei-Feng Xuan, Yan-Kui Song, Notes on selectively $2$-star-ccc spaces. RACSAM. 114 (2020) 155.
		
		
		
		
	\end{thebibliography}
\end{document}